\documentclass[10pt]{article}

\usepackage{bigints}

\usepackage{amssymb}
\usepackage{mathrsfs}
\usepackage{amsmath}
\usepackage{amsfonts}

\renewcommand{\thefootnote}{\fnsymbol{footnote}}

\usepackage{amsthm}
\newtheorem{theo}{Theorem}[section]
\newtheorem{prop}{Proposition}[section]
\newtheorem{coro}{Corollary}[section]
\newtheorem{lemm}{Lemma}[section]

\newtheorem{main}{Main Theorem}

\theoremstyle{definition}

\theoremstyle{remark}

\newtheorem{rema}{Remark}[section]
\newtheorem{ques}{Question}[section]

\usepackage{hyperref}

\title{Upper bound preservation of the total scalar curvature in a conformal class}
\author{Shota Hamanaka\thanks{The author's present affiliation is Department of Mathematics, Graduate School of Pure and Applied Sciences, University of Tsukuba.}}
\date{\today}

\begin{document}
\maketitle
 
\renewcommand{\thefootnote}{\fnsymbol{footnote}} 
\renewcommand{\thefootnote}{\arabic{footnote}}

\renewcommand{\thefootnote}{\fnsymbol{footnote}} 
\renewcommand{\thefootnote}{\arabic{footnote}}
  
\begin{abstract}
We show that in an arbitrarily fixed conformal class with non-positive Yamabe constant on a closed manifold, the upper bound of the total scalar curvature is preserved under the $C^{0}$-convergence of metrics provided that the convergent sequence has a uniform H\"{o}lder bound.
Moreover, if we consider the condition that the scalar curvature is bounded by some fixed continuous function from below in addition to the upper bound of the total scalar curvature, then such a condition is $C^{0}$-closed in the intersection of an arbitrarily fixed positive Yamabe conformal class and the space of metrics with a uniform H\"{o}lder bound.
\end{abstract}

\section{Introduction}\label{sec1}
Let $M^{n}$ be a smooth manifold of dimension $n \ge 2$ and $\mathcal{M}$ the space of all $C^{2}$-Riemannian metrics on $M.$
For example, the followings are known for the topology of several subspaces of $\mathcal{M}$.
\begin{itemize}
\item Gromov \cite{gromov2014dirac} and Bamler \cite{bamler2016proof} proved that for any continuous function $\sigma \in C^{0}(M),$ the space
\[
\left\{ g \in \mathcal{M}~\middle|~R(g) \ge \sigma \right\}
\]
is closed in $\mathcal{M}$ with respect to the $C^{0}$-topology,
i.e., if a sequence of $C^{2}$-metrics $g_{i}$ converges to a $C^{2}$-metric $g$ in the local $C^{0}$-sense on $M$ that satisfy $R(g_{i}) \ge \sigma,$
then $R(g) \ge \sigma.$
Here, $R(g)$ denotes the scalar curvature of $g$.
\item Lohkamp \cite{lohkamp1995curvature} proved that for any continuous function $\sigma \in C^{0}(M),$ the space
\[
\mathcal{R}(M)_{\le \sigma} := \left\{ g \in \mathcal{M}~\middle|~R(g) \le \sigma \right\}
\]
is dense in $\mathcal{M}$ with respect to the $C^{0}$-topology,
i.e., for any $C^{2}$-metric $g,$ there exists a sequence of $C^{2}$-metrics $(g_{i})$ satisfying $R(g_{i}) \le \sigma$ such that $g_{i}$ converges to $g$ in the locally $C^{0}$-sense on $M$.

This result also implies that, when $M$ is closed (i.e., compact without boundary), for any constant $\kappa \in \mathbb{R},$ the space
\[
\left\{ g \in \mathcal{M}~\middle|~\int_{M} R(g)\, d\mathrm{vol}_{g} \le \kappa \right\}
\]
is dense in $\mathcal{M}$ with respect to the $C^{0}$-topology.
Here, $d\mathrm{vol}_{g}$ denotes the Riemannian volume measure of $g.$
\begin{proof}
  From the above fact by Lohkamp, we can take a sequence of metrics $(g_{i})$ such that
\[
R(g_{i}) \le \frac{\kappa}{2 \mathrm{Vol}(M, g)}~~\mathrm{if}~\kappa \ge 0~~\left(\mathrm{resp.}~R(g_{i}) \le \frac{2 \kappa}{\mathrm{Vol}(M, g)}~~\mathrm{if}~\kappa < 0 \right),
\]
and $g_{i}$ converges to $g$ in the $C^{0}$-sense on $M.$
Integrating both sides with respect to $g_{i}$, we get
\[
\int_{M} R(g_{i})\, d\mathrm{vol}_{g_{i}} \le \frac{\mathrm{Vol}(M, g_{i})}{2 \mathrm{Vol}(M, g)} \kappa~~\mathrm{if}~\kappa \ge 0~~\left( \mathrm{resp.}~\le \frac{2 \mathrm{Vol}(M, g_{i})}{\mathrm{Vol}(M, g)} \kappa~~\mathrm{if}~\kappa < 0 \right),
\]
Since $g \mapsto \mathrm{Vol}(M, g)$ is continuous with respect to $C^{0}$-topology, hence $\mathrm{Vol}(M, g_{i}) / \mathrm{Vol}(M, g) \rightarrow 1$ as $i \rightarrow \infty.$
Therefore there is a sufficiently large $i_{0} \in \mathbb{N}$ such that for all $i \ge i_{0}$, $\int_{M} R(g_{i})\, d\mathrm{vol}_{g_{i}} \le \kappa$, and hence $(g_{i})_{i \ge i_{0}}$ is the desired $C^{0}$-approximation of $g$.
\end{proof}
This Lohkamp's approximation was built from certain $C^{0}$-deformation which is not just a conformal deformation.
Indeed, in the present paper, we will show that in a fixed conformal class, such a subspace is $C^{0}$-closed in some sense (see Corollary \ref{coro1} below).
\item Lohkamp \cite{besson1996riemannian} proved that for any continuous function $\sigma \in C^{0}(M),$ both of the spaces
\[
\left\{ g \in \mathcal{M}~\middle|~\mathrm{Ric}(g) \le \sigma \cdot g \right\}
\]
and
\[
\left\{ g \in \mathcal{M}~\middle|~R(g) \le \sigma \right\}
\]
are closed in $\mathcal{M}$ with respect to $C^{1}$-topology.
\item As is well known to experts, the Einstein--Hilbert functional:
\[
\mathcal{M} \ni g \mapsto \int_M R_g\, d\mathrm{vol}_g
\]
is continuous with respect to $C^0 \cap W^{1,2}$-topology on a fixed closed manifold (see Fact 1.2 in author's preprint, arXiv:2208.01865v18). 
On the other hand, 
the author (in the same preprint) especially proved that when $M^n$ is closed $n$-manifold, for any continuous nonnegative function $\sigma \in C^{0}(M, \mathbb{R}_{\ge 0})$, constants $\kappa \in \mathbb{R}$ and $\Lambda > 0$, the space
\[
\left\{ (g, f) \in \mathcal{M} \times \mathrm{Lip}^{\Lambda}(M)~\middle|~\int_{M} R_{g, f}e^{-f}\, d\mathrm{vol}_{g} \ge \kappa,~R_g \ge \sigma \right\}
\]
is closed in $\mathcal{M} \times \mathrm{Lip}^{\Lambda}(M)$ with respect to the product topology of $W^{1, p}~(p > n^2/2)$ and $C^0$-topology.
Here, $\mathrm{Lip}^{\Lambda}(M)$, $R_{g, f}$ denote the space of $\Lambda$-Lipschitz functions on $M$ and the $P$-scalar curvature: $R_{g, f} = R_g + 2\, \Delta_g f - \| \nabla f \|_g^2$,
respectively.
\end{itemize}
In the present paper, we will examine the space of all Riemannian metrics whose total scalar curvatures are bounded from \textbf{above} by some constant.
Our first main result is the following.
\begin{main}
\label{main}
Let $M^{n}$ be a closed manifold of dimension $n \ge 3$ and $g_{0}$ a $C^{2}$-Riemannian metric on $M.$
Let $u_{i}, u : M \rightarrow \mathbb{R}_{+}~(i = 1,2, \cdots )$ be positive $C^{2}$-functions on $M$
and assume the following: 
\begin{itemize}
\item[$(0)$] $\| u_i \|_{C^{\alpha}(M)} \le D$ for some $\alpha \in (0,1)$ and a positive constant $D > 0$,
\item[$(1)$] $g_{i} := u_{i}^{\frac{4}{n-2}} g_{0} \overset{C^{0}}{\longrightarrow} g := u^{\frac{4}{n-2}} g_{0}$ on $M,$
\item[$(2)$] there is a constant $\kappa \in \mathbb{R}$ such that $\int_{M} R(g_{i})\, d\mathrm{vol}_{g_{i}} \le \kappa~~\mathrm{for~all}~i.$
\end{itemize}
When $(M, g_{0})$ is Yamabe positive, i.e., $Y(M, g_{0}) > 0$ (see Section \ref{section-proof} for the definition of the Yamabe constant $Y(M, g_{0})$), we additionally assume that 
\begin{itemize}
\item[$(3)$] there is a continuous function  $\delta \in C^{0}(M)$ such that for all $i$, $R(g_{i}) \ge \delta$ on $M.$
\end{itemize}
Then $\int_{M} R(g)\, d\mathrm{vol}_{g} \le \kappa.$
Here, $R(g),$ $d\mathrm{vol}_{g}$ denote respectively the scalar curvature and the Riemannian volume measure of $g.$
Moreover, when $(M, g_{0})$ is Yamabe positive, the same assertion still holds even if $(1)$ is replaced with the following weaker condition:
\begin{itemize}
\item[$(1)'$] $u_{i} \rightarrow u$ in the $L^{\frac{2n}{n-2}}(M, g_{0})$-sense, i.e., 
\[
\int_{M} \left| u_{i} - u \right|^{\frac{2n}{n-2}} d\mathrm{vol}_{g_{0}} \rightarrow 0~~~\mathrm{as}~i \rightarrow \infty.
\]
\end{itemize}
\end{main}
As a corollary of Main Theorem \ref{main} combined with {\cite[p.1118]{gromov2014dirac}}, we obtain the following.
\begin{coro}
\label{coro1}
Let $M^{n}$ be a closed manifold of dimension $n \ge 2$ and $g_{0}$ a $C^{2}$-Riemannian metric on $M.$
Let $\mathcal{M}$ be the space of all $C^{2}$-Riemannian metrics on $M.$
Then, for any continuous function $\sigma \in C^{0}(M)$, constants $\kappa \in \mathbb{R}$ and $D > 0$, the space
\[
\left\{ g \in [g_{0}]~\middle|~\int_{M} R(g)\, d\mathrm{vol}_{g} \le \kappa,~R(g) \ge \sigma,~\| g \|_{C^{\alpha}(M, g_0)} \le D \right\}~\left( \subset \mathcal{M} \right)
\]
when $Y(M, g_{0}) > 0,$ and the space 
\[
\left\{ g \in [g_{0}]~\middle|~\int_{M} R(g)\, d\mathrm{vol}_{g} \le \kappa,~\| g \|_{C^{\alpha}(M, g_0)} \le D \right\}~\left( \subset \mathcal{M} \right)
\]
when $Y(M, g_{0}) \le 0$ are both closed in $\mathcal{M}$ with respect to $C^{0}$-topology.
Here, $[g_{0}] := \{ g = u \cdot g_{0} \in \mathcal{M}~|~u \in C^{2}(M),~u > 0~\mathrm{on}~M \}$ is the ($C^{2}$-) conformal class of $g_{0}.$
\end{coro}
When $n = 2,$ the above corollary is a direct conclusion from the Gauss-Bonnet Theorem and {\cite[p.1118]{gromov2014dirac}}.
In the case of $\kappa \le 0$ (in particular, $Y(M, g_{0}) \le 0$), we can refine Main Theorem \ref{main} to the following one with a slightly weaker assumption. 
\begin{main}
\label{main2}
Let $M^{n}$ be a closed manifold of dimension $n \ge 3$ and $g_{0}$ a $C^{2}$-Riemannian metric on $M.$
Let $u_{i}, u : M \rightarrow \mathbb{R}_{+}~(i = 1,2, \cdots )$ be positive $C^{2}$-functions on $M$
and set $g_{i} := u_{i}^{\frac{4}{n-2}} g_{0}$ and $g := u^{\frac{4}{n-2}} g_{0}.$
Assume the following:
\begin{itemize}
\item[$(a)$] there is a positive constant $C_{0}$ and $\alpha \in (0,1)$ such that $C_{0}^{-1} \le u_{i} \le C_{0}$ on $M$ and $\| u_i \|_{C^{\alpha}(M)} \le C_0$, for all $i$,
\item[$(b)$] $u_{i} \rightarrow u$ in the $L^{1}(M, g_{0})$-sense, i.e., 
\[
\int_{M} |u_{i} - u|\, d\mathrm{vol}_{g_{0}} \rightarrow 0~~~\mathrm{as}~i \rightarrow \infty,
\]
\item[$(c)$] there is a non-positive constant $\kappa \in \mathbb{R}$ such that $\int_{M} R(g_{i})\, d\mathrm{vol}_{g_{i}} \le \kappa \le 0$ for all $i$.
In particular, this assumption implies that $g_{0}$ is Yamabe non-positive (i.e., $Y(M, g_{0}) \le 0$).
\end{itemize}
Then $\int_{M} R(g)\, d\mathrm{vol}_{g} \le \kappa.$
\end{main}

This paper is organized as follows.
In Section 2, we prove Main Theorem \ref{main}. 
The key fact to prove Main Theorem \ref{main} is that the normalized Yamabe flow starting from each $g_{i}$ subconverges to the normalized Yamabe flow starting from the limiting metric $g.$
This has been proven separately in two cases: Yamabe non-positive (i.e., $Y(M, g_{0}) \le 0$) and positive (i.e., $Y(M, g_{0}) > 0$).
In Section 3, we prove Main Theorem \ref{main2}.
In the setting of Main Theorem \ref{main2}, we observe the unnormalized Yamabe flow starting from each $g_{i}.$
We prepare an $L^{1}$-estimate for such flows and use it to prove that $g_{i}$ actually converges to $g$ in the $C^{0}$-sense on $M$ as $i \rightarrow \infty.$
Then Main Theorem \ref{main2} follows from Main Theorem \ref{main}. 

\subsection*{Acknowledgements}
\,\,\,\,\,\,\, This work was supported by Foundation of Research Fellows, The Mathematical Society of Japan.
The author was also supported by JSPS KAKENHI Grant Number 24KJ0153.
The author would like to thank J{\o}rgen Olsen Lye for notifying him of their paper \cite{carron2021convergence}.

\section{Proof of Main Theorem \ref{main}}
\label{section-proof}

\,\,\,\,\,\,\, Let $M$ be a closed (i.e., compact without boundary) manifold of dimension $n \ge 3$
and let $g$ be a Riemannian metric on $M.$
We will denote the Riemannian volume measure of $(M, g)$ as $d\mathrm{vol}_{g}.$
We consider the (volume) normalized Yamabe flow:
\begin{equation}\label{eq-normalized-yamabe}
  \frac{\partial}{\partial t} g(t) = - \left( R(g(t)) - r(g(t)) \right)\, g(t),
\end{equation}
where $R(g(t))$ is the scalar curvature of $g(t)$ and $r(g(t))$ is the mean value of $R(g(t)),$ i.e.,
\[
r(g(t)) = \frac{\int_{M} R(g(t))\, d\mathrm{vol}_{g(t)}}{\int_{M} d\mathrm{vol}_{g(t)}}.
\]
Since 
\[
\begin{split}
\frac{d}{dt} \mathrm{Vol}(M, g(t)) &= \int_{M} \frac{1}{2} \mathrm{tr}_{g(t)} \left( \frac{\partial}{\partial t} g(t) \right)\, d\mathrm{vol}_{g(t)} \\
&= -\frac{n}{2} \int_{M} \left( R(g(t)) - r(g(t)) \right)\, d\mathrm{vol}_{g(t)}
= 0,
\end{split}
\]
the volume $\mathrm{Vol}(M, g(t))$ is invariant along the normalized Yamabe flow.
The Yamabe constant of a Riemannian metric $g_{0}$ is defined as the infimum of the normalized Einstein-Hilbert functional among all metrics conformally equivalent to $g_{0},$ i.e., 
\[
Y(M, g_{0}) := \inf \left\{ \frac{\int_{M} R(g)\, d\mathrm{vol}_{g}}{\left( \int_{M} d\mathrm{vol}_{g} \right)^{\frac{n-2}{n}}}~\middle|~g = u^{\frac{4}{n-2}} g_{0},~u \in C^{\infty}(M),~u > 0~\mathrm{on}~M \right\}.
\]
By the definition, $Y(M, g_{0})$ depends only on the conformal class $[g_{0}]$ of $g_{0}.$

Since the normalized Yamabe flow preserves the conformal structure, we may write $g(t) = u(t)^{\frac{4}{n-2}} g_{0},$ where $g_{0}$ is a fixed background metric on $M$ and $u(t)$ is a positive function.
The scalar curvature of $g(t)$ is related to the scalar curvature of $g_{0}$ by
\[
R(g(t)) = - u^{-\frac{n+2}{n-2}} \left( \frac{4(n-1)}{n-2} \Delta_{g_{0}} u - R(g_{0}) u \right).
\]
Hence, the normalized Yamabe flow reduces to the following evolution equation for the conformal factor can be written as 
\[
\frac{\partial}{\partial t} u(t) = \frac{n+2}{4} \left( \frac{4(n-1)}{n-2}\Delta_{g_{0}}u(t)^{\frac{n-2}{n+2}} - R(g_{0}) u(t)^{\frac{n-2}{n+2}} + r(g(t)) u(t) \right).
\]
Using the identity,
\[
\frac{\partial}{\partial t} R(g(t)) = (n-1) \Delta_{g(t)} R(g(t)) + R(g(t)) \left( R(g(t)) - r(g(t)) \right), 
\]
we obtain that
\[
\frac{d}{dt} r(g(t)) = -\frac{n-2}{2} \mathrm{Vol}(M, g(t))^{-1} \int_{M} \left( R(g(t)) -r(g(t)) \right)^{2}\, d\mathrm{vol}_{g(t)}.
\]
In particular, the function $t \mapsto r(g(t))$ is decreasing.

It is well-known that for a given positive function $u_{0}$ on $M$ with sufficiently high regularity, the normalized Yamabe flow always has a unique positive smooth solution $u(t)$ up to a positive time $T > 0$ with initial data $u_{0}.$
We use the following version by Carron, Lye and Vertman \cite{carron2021convergence}.
\begin{prop}[{\cite[Theorem 2.7, Theorem 2.9]{carron2021convergence}}]
\label{prop-short}
Let $M^{n}$ be a closed manifold of dimension $n \ge 3$ and $p > n/2.$
Let $g_{0}$ be a $C^{2}$-Riemannian metric on $M$ and
let $u_{0}$ be a positive $C^{2}$-function.
Then there is a positive time $T > 0$ such that a unique positive solution $u(t)$ of the normalized Yamabe flow (\ref{eq-normalized-yamabe}) on $M \times [0, T)$, and   the solution $u(t)$ is smooth on $M \times (0, T)$.
Moreover, $u(t) \rightarrow u_{0}$ as $t \rightarrow 0$ in the $C^{0}$-sense
and 
$R(g(t)) = R ( u(t)^{\frac{4}{n-2}} g_{0} )$ converges to $R( u_{0}^{\frac{4}{n-2}} g_{0} )$ as $t \rightarrow 0$
in the $C^{0}$-sense.
\end{prop}
\begin{proof}
Since $(M, g_{0})$ is a closed smooth Riemannian manifold, the assumption of {\cite[Theorem 2.7]{carron2021convergence}} is satisfied.
Hence the assertion follows from {\cite[Theorem 2.9]{carron2021convergence}}.
Note that Assumption 1 in \cite{carron2021convergence} is used to gain the Sobolev inequality.
However, in our setting, the manifold $(M, g_{0})$ is smooth and closed, hence the Sobolev inequality holds regardless of the sign of the Yamabe constant.
As a result, we can obtain the unique short-time solution of (\ref{eq-normalized-yamabe}) for a closed smooth Riemannian manifold $(M, g_{0})$ with $Y(M, g_{0}) \le 0$ as well.
Finally, the last assertion follows from {\cite[The annotation 6 in Theorem 2.9]{carron2021convergence}}.
\end{proof}
\begin{rema}
\label{rema-prop2.1}
Since 
$u(t) \rightarrow u_{0}$ and 
$R(g(t)) \rightarrow R \left( u_{0}^{\frac{4}{n-2}} g_{0} \right)$ in the $C^{0}$-sense as $t \rightarrow 0$, 
from the mean value theorem, we still have that 
$t \mapsto r(g(t))$ is decreasing on $[0, T)$.
\end{rema}
The existence of a long-time solution were settled by Hamilton (in his unpublished manuscript), Chow \cite{chow1992yamabe}, Schwetlick-Struwe \cite{schwetlick2003convergence} and Brendle \cite{brendle2005convergence}.
\begin{theo}[\cite{chow1992yamabe, schwetlick2003convergence, brendle2005convergence},~cf. {\cite[Proof of Proposition 3.6]{carron2021convergence}}]
Let $M, g_{0}$ and $u_{0}$ be the same as those in the previous proposition.
Then the positive solution $u(t)$ as in Proposition \ref{prop-short} exists for all time (i.e., $T = \infty$).
\end{theo}
Therefore, in the setting of Main Theorem \ref{main}, we can take a uniform positive time $T > 0$ such that
the unique positive solutions $u(t), u_{i}(t)$ of the normalized Yamabe flow (\ref{eq-normalized-yamabe}) with the initial functions $u, u_{i}$ respectively
exists up to $T.$
Hereafter, we fix such a positive existence time $T > 0$ (independent of $i$). 

From the assumption $(1)$ in Main Theorem \ref{main}, by taking $i$ large enough, we can assume that 
\[
\frac{1}{2}\, \mathrm{Vol}(M, g) \le \mathrm{Vol}(M, g_{i}) \le 2\, \mathrm{Vol}(M, g)~~\mathrm{for~all}~i,
\]
where $\mathrm{Vol}(M, h)$ is the volume of $M$ measured by the Riemannian metric $h.$
Moreover, since the normalized Yamabe flow preserves the volume, we obtain that
\[
\frac{1}{2}\, \mathrm{Vol}(M, g) \le \mathrm{Vol}(M, g_{i}(t)) \le 2\, \mathrm{Vol}(M, g)
\]
for all $i$ and $t \in [0, T]$ as well.
\begin{rema}
\label{rema-volume}
Since $g_{i} = u_{i}^{\frac{4}{n-2}} g_{0}$ and $g = u^{\frac{4}{n-2}} g_{0}$ respectively, we have respectively
\[
\mathrm{Vol}(M, g_{i}) = \int_{M} u_{i}^{\frac{2n}{n-2}}\, d\mathrm{vol}_{g_{0}}~\mathrm{and}~
\mathrm{Vol}(M, g) = \int_{M} u^{\frac{2n}{n-2}}\, d\mathrm{vol}_{g_{0}}.
\]
Hence the above volume estimates also hold under the weaker assumption $(1)'$ in Main Theorem \ref{main}.
\end{rema}
We will establish the convergence of $u_{i}(t)$ to $u(t)$ on $M \times [0, T]$ below.
We shall do this in two cases: $Y(M, g_{0}) \le 0$ and $Y(M, g_{0}) > 0.$

\medskip
\noindent
\underline{\textbf{Yamabe non-positive case (i.e., $Y(M, g_{0}) \le 0$)}}:
As stated in {\cite[Section~3]{ye1994global}}, since $Y(M, g_{0}) < 0~(\mathrm{resp.}~\le 0),$
we can choose a background metric $\tilde{g}_{0} \in [g_{0}]$ such that $R(\tilde{g}_{0}) < 0~(\mathrm{resp.}~\le 0)$ on $M.$
Set $\tilde{g}_{0} =: \tilde{u}^{-\frac{4}{n-2}} g_{0},$
then $u \cdot \tilde{u}$ and $u_{i} \cdot \tilde{u}$ satisfy the assumption $(1)$ in Main Theorem \ref{main}.
Thus, without loss of generality, we can assume that $R(g_{0}) < 0~(\mathrm{resp.}~\le 0)$ on $M.$
From the maximum principle {\cite[(3.1)]{ye1994global}}, we obtain that for all $t \in (0, T),$
\[
\frac{d}{dt} u_{i, \mathrm{min}}^{\frac{n+2}{n-2}}(t) \ge \frac{n-2}{4(n-1)} \min_{M} \left| R(g_{0}) \right|\, u_{i, \mathrm{min}}(t) + \frac{n-2}{4(n-1)} r(g_{i}(t))\, u_{i, \mathrm{min}}^{\frac{n+2}{n-2}}(t),
\] 
where $u_{i, \mathrm{min}}(t) := \min_{M} u_{i}(\cdot, t).$
By the definition of $Y(M, g_{0})$ and the above volume estimate (see also Remark \ref{rema-volume}), we have
$r(g_{i}(t)) \ge Y(M, g_{0}) \mathrm{Vol}(M, g_{i}(t))^{-\frac{2}{n}} \ge \frac{1}{2} Y(M, g_{0}) \mathrm{Vol}(M, g)^{-\frac{2}{n}}.$
Thus we get
\[
\frac{d}{dt} u_{i, \mathrm{min}}^{\frac{n+2}{n-2}}(t) \ge \frac{n-2}{8(n-1)} Y(M, g_{0}) \mathrm{Vol}(M, g)^{-\frac{2}{n}}\, u_{i, \mathrm{min}}^{\frac{n+2}{n-2}}(t)
\]
for all $t \in (0,T).$
Hence, from this estimate and the mean value theorem (cf. Remark \ref{rema-prop2.1}), we obtain that
\[
u_{i, \mathrm{min}}^{\frac{n+2}{n-2}}(t) \ge \exp \left( \frac{n-2}{8(n-1)} Y(M, g_{0}) \mathrm{Vol}(M, g)^{-\frac{2}{n}}\, t \right) \cdot u_{i, \mathrm{min}}^{\frac{n+2}{n-2}}(0)
\]
for all $t \in [0, T].$
On the other hand, the maximum principle {\cite[(3.3)]{ye1994global}} implies that
\[
\frac{d}{dt} u_{i, \mathrm{max}}^{\frac{n+2}{n-2}}(t) \le -\frac{n-2}{4(n-1)} \left( \min_{M} R(g_{0}) \right)\, u_{i, \mathrm{max}}(t) + \frac{n-2}{4(n-1)} r(g_{i}(t))\, u_{i, \mathrm{max}}^{\frac{n+2}{n-2}}(t),
\]
where $u_{i, \mathrm{max}} (t) := \max_{M} u_{i}(\cdot, t).$
This implies that
\[
\frac{d}{dt} u_{i, \mathrm{max}}^{\frac{4}{n-2}} \le - \frac{n-2}{(n-1)(n+2)} \left( \min_{M} R(g_{0}) \right) + \frac{n-2}{(n-1)(n+2)} r(g_{i}(t)) u^{\frac{4}{n-2}}_{i, \mathrm{max}}(t).
\]
Moreover, since $t \mapsto r(g_{i}(t))$ is decreasing, we have
\[
\frac{d}{dt} u_{i, \mathrm{max}}^{\frac{4}{n-2}} \le - \frac{n-2}{(n-1)(n+2)} \left( \min_{M} R(g_{0}) \right) + \frac{n-2}{(n-1)(n+2)} r(g_{i}(0)) u^{\frac{4}{n-2}}_{i, \mathrm{max}}(t).
\]
By Gronwall's inequality, we obtain that
\[
u_{i, \mathrm{max}}^{\frac{4}{n-2}}(t) \le \alpha(t) + \int^{t}_{0} \alpha(s) \beta\,\exp \left( \beta (t-s) \right)\, ds~~\mathrm{for~all}~t \in [0,T],
\]
where
\[
\alpha(t) := u_{i, \mathrm{max}}(0)^{\frac{4}{n-2}} - \frac{n-2}{(n-1)(n+2)} \left( \min_{M} R(g_{0}) \right) t
\]
and
\[
\beta := \frac{n-2}{(n-1)(n+2)} \max \left\{ r(g_{i}(0)), 0 \right\}.
\]
From the assumption $(2)$ in Main Theorem \ref{main} and the lower volume estimate (Remark \ref{rema-volume}),
\[
\beta \le \frac{2(n-2)}{(n-1)(n+2)} \mathrm{Vol}(M, g)^{-1} \cdot \max \{ \kappa, 0 \} := \tilde{\beta}.
\]
Hence we obtain that
\[
u_{i, \mathrm{max}}^{\frac{4}{n-2}}(t) \le \alpha(t) + \int^{t}_{0} \alpha(s) \tilde{\beta}\, \exp \left( \tilde{\beta} (t-s) \right)\, ds~~\mathrm{for~all}~t \in [0,T].
\] 
Therefore there is a positive constant $C$ depending only on $\kappa, u, g_{0}$ and $M$
such that
\[
0 < C^{-1} \le u_{i}(x, t) \le C~~\mathrm{for~all}~(x, t) \in M \times [0,T]. 
\]
Unless otherwise stated, we will denote positive constants depending only on the given datum ($D, \kappa, u, g_{0}$ and $M$) as the same symbol $C.$
Then, from Krylov-Safonov estimate (\cite{krylov1981certain}, see also {\cite[Theorem~12]{picard2019notes}} and {\cite[Proposition~4.2]{caldeira2021normalized}}), we obtain the following $C^{\alpha}$-estimate for $\alpha \in (0, 1)$, which is in our assumption:
\begin{equation}\label{eq-C-alpha}
|| u_{i}(\cdot, \cdot) ||_{C^{\alpha/2, \alpha} \left (M \times [t_{0}, T] \right)} \le C~~\mathrm{for~all}~t_{0} \in (0,T)~\mathrm{and}~i.
\end{equation}
Here, the constant $C$ does not depend on time $t_0$ and $i$.
From this estimate, the fact that the constant $C$ is independent of $t_0$ and $i$, the regularity assumption of Main Theorem \ref{main} and the interior parabolic Schauder estimate ({\cite[Proposition~4.2.]{caldeira2021normalized}}), we obtain that
\begin{equation}\label{eq-C-2alpha}
|| u_{i}(\cdot, \cdot) ||_{C^{(2+\alpha')/2, 2 + \alpha'} \left( M \times [t, T] \right)} \le C~~\mathrm{for~all}~i~\mathrm{and}~t \in (0, T),
\end{equation}
where $\alpha' = \frac{n-2}{n+2} \alpha.$
Note here that we have used the fact that $t \mapsto r(g(t))$ is decreasing and the lower volume estimate (Remark \ref{rema-volume}) when we applied the Schauder estimate. 
On the other hand, since {\cite[Theorem 2.7]{carron2021convergence}} also implies $u_{i}(t) \rightarrow u_{i}(0)$ in $W^{2,p}(M, g_{0})$ for all $p > n/2$ as $t \rightarrow 0$,
by the Sobolev embedding: $W^{2,p} \hookrightarrow C^{0, \alpha}$ for sufficiently large $p >> n/2,$
we obtain that
$u_{i}(t) \rightarrow u_{i}(0)$ in $C^{0, \alpha}(M, g_{0})$ as $t \rightarrow 0$.
Combining these things, (\ref{eq-C-alpha}) and our uniform H\"{o}lder bound assumption $(0)$, we obtain
\begin{equation}\label{eq-C-alpha-upto0}
|| u_{i}(\cdot, \cdot) ||_{C^{\alpha/2, \alpha} \left (M \times [0, T] \right)} \le C
\end{equation}
for some positive constant $C$ independent of $i$.
From (\ref{eq-C-alpha-upto0}) and (\ref{eq-C-2alpha}), we can choose a subsequence $(u_{i_{k}}(\cdot, t))_{t \in [0, T]}$ such that 
$u_{i_{k}}$ converges to some $C^{2}$ function $\tilde{u}$ on $M \times (0,T]$ in the locally uniformly $C^{2}$-sense in $M \times (0,T]$ and $\tilde{u}$ satisfies the normalized Yamabe flow equation (\ref{eq-normalized-yamabe}) there.
Moreover, from the $C^{\alpha}$-estimate (\ref{eq-C-alpha-upto0}) and the standard diagonal argument, we can choose the subsequence $u_{i_{k}}$ so that the limiting function $\tilde{u}$ is a $C^{2}$ function in $M \times (0,T]$ and a $C^{\beta}$ function on $M \times [0,T]$ for some $\beta < \alpha$, and $u_{i_{k}}$ converges to $\tilde{u}$ in the $C^{\beta}$-sense on $M \times [0,T]$.
In particular, since $u_{i_{k}}(\cdot, 0) \rightarrow u(\cdot)$ as $k \rightarrow \infty$
(from the assumption $(1)$ in Main Theorem \ref{main}),
$\tilde{u}(\cdot, 0) \equiv u(\cdot).$
Then, from the uniqueness of the normalized Yamabe flow ({\cite[Corollary~3.3]{caldeira2021normalized}}), we finally obtain that
$\tilde{u}(t) \equiv u(t)$ for all $t \in [0,T].$ 
\begin{rema}\label{rema-uniqueness}
  Our situation here and the one in {\cite[Corollary~3.3]{caldeira2021normalized}} are different.
  However, one can still apply the same arguments in {\cite[Corollary~3.3]{caldeira2021normalized}} to our setting and can obtain the uniqueness since $\tilde{u}(\cdot, 0) \equiv u(\cdot)$ and $\tilde{u},~u$ both satisfy the unnormalized Yamabe flow equation (\ref{eq-normalized-yamabe}) in $M \times (0, T]$. 
  More precisely, to establish this uniqueness, it suffices that the Yamabe equation is satisfied in the interior of the space-time cylinder and that the solution converges continuously to the same initial data.
\end{rema}

\medskip
\noindent
\underline{\textbf{Yamabe positive case (i.e., $Y(M, g_{0}) > 0$)}}:
As done in the previous case, we can assume that $R(g_{0}) > 0$ without loss of generality.
Since the assumption $(1)'$ of Main Theorem \ref{main} is weaker than $(1),$
it is enough to prove the theorem under the assumption $(1)'.$ 
Then we can firstly obtain the following.
\begin{prop}[{\cite[Proposition~2.1.]{brendle2005convergence}}]
\label{prop-pres}
The scalar curvature of the metric $g_{i}(t)$ satisfies
\[
\inf_{M} R(g_{i}(t)) \ge \min \left\{ \delta, 0 \right\}~~\mathrm{for~all}~t \ge 0.
\]
\end{prop}
\begin{proof}
Since $g_{i}(t)$ satisfies the normalized Yamabe flow (\ref{eq-normalized-yamabe}) in $M \times (0,T),$ from {\cite[Proposition 2.1]{brendle2005convergence}}, 
\[
\inf_{M} R(g_{i}(t)) \ge \min \left\{ \delta, 0 \right\}~~\mathrm{for~all}~t \in (0, T).
\]
Then, the assertion follows from the $C^{0}$-continuity: $R(g_{i}(t)) \rightarrow R(g_{i})$ stated in Remark \ref{rema-prop2.1}.
\end{proof}
Set
\[
\sigma := \max \left\{ 1-\delta, 1 \right\},
\]
then $R(g_{i}(t)) + \sigma \ge 1$ for all $t \ge 0.$
Moreover, Brendle {\cite[Proposition~2.4]{brendle2005convergence}} also gave upper and lower bounds of $u_{i}(t).$
\begin{prop}[{\cite[Proposition~2.4]{brendle2005convergence}}]
There are positive constants $C$ and $c$ depending only on
$T, \delta, \kappa, g$ and $M$ such that 
\[
\sup_{M} u_{i}(t) \le C
\]
and
\[
\inf_{M} u_{i}(t) \ge c
\]
for all $t \in [0, T].$
\end{prop} 
\begin{proof}
The following proof mostly follow the argument of Brendle {\cite[Proposition~2.4]{brendle2005convergence}}.
The function $u_{i}(t)$ ($t \in (0, T)$) satisfies 
\[
\begin{split}
\frac{\partial}{\partial t} u_{i}(t) &= -\frac{n-2}{4} \left( R(g_{i}(t)) - r(g_{i}(t)) \right) \\
&\le \frac{n-2}{4} (r(g_{i}(0)) + \sigma) \\
&\le \frac{n-2}{4} \left( \frac{1}{2} \mathrm{Vol}(M, g)^{-1} \kappa + \sigma \right).
\end{split}
\]
We have used the assumption $(2)$ of Main Theorem \ref{main} and the volume lower estimate (Remark \ref{rema-volume}) in the last inequality.
From this estimate, the continuity of $u_{i}(t)$ as $t \rightarrow 0$ (see Remark \ref{rema-prop2.1}) and the mean value theorem, we obtain that
\[
\sup_{M} u_{i}(t) \le C~~\mathrm{for~all}~t \in [0, T].
\]
Moreover from this estimate, we follow the same argument in {\cite[Proposition~2.4]{brendle2005convergence}} and obtain the uniformly lower bound on $M \times [t_{0}, T]$ for all $t_{0} \in (0, T).$
Thus, we can obtain the lower estimate on the whole space-time $M \times [0,T]$ by the continuity of $u_{i}(t)$ as $t \rightarrow 0$ (see Remark \ref{rema-prop2.1}).
\end{proof}
\begin{rema}
  In order to obtain the uniform lower bound of $\inf_{M} u_{i}$ here, we have specifically used the fact that the volume $\mathrm{Vol}(M, u_{i}(t)^{\frac{4}{n-2}} g_{0})$ is bounded from below by some positive constant independent of time. Especially in this case, it is constant in time.
\end{rema}
Hence, from the same argument in the previous case, we can obtain that there is a function $\tilde{u}$, which is $C^{2}$ in $M \times (0,T]$ and $C^{\beta}$ for some $\beta < \alpha$ on $M \times [0,T]$ such that $u_{i}$ subconverges to $\tilde{u}$ in the locally uniformly $C^{2}$-sense in $M \times (0,T]$ and $C^{\beta}$-sense on $M \times [0,T]$.
From the assumption $(1)'$ in Main Theorem \ref{main}, as $i \rightarrow \infty,$ we obtain that 
\[
\int_{M} \left| \tilde{u}(\cdot, 0) - u \right|^{\frac{2n}{n-2}} d\mathrm{vol}_{g_{0}} = 0.
\]
Hence we obtain that $\tilde{u}(\cdot, 0) \equiv u.$
Indeed, if not, there is a point $p \in M$ such that $\tilde{u}(\cdot, 0)(p) - u(p) \neq 0.$
Without loss of generality, we may assume that $\tilde{u}(\cdot, 0)(p) - u(p) := \varepsilon > 0.$
Since $\tilde{u}(\cdot, 0) - u$ is continuous on $M,$ there are sufficiently small open neighborhoods $U, V$ of $p$ such that
$\tilde{u}(\cdot, 0) - u \ge \varepsilon / 2$ on $V$ and the closure of $V$ is contained in $U.$
Let $\psi \in C^{\infty}(M)$ be a smooth cut-off function such that
$0 \le \psi \le 1,$
$\psi \equiv 1$ on $V$ and  $\psi \equiv 0$ on $M \setminus U.$
Then, 
\[
\begin{split}
0 &= \int_{M} \left| \tilde{u}(\cdot, 0) - u \right|^{\frac{2n}{n-2}} d\mathrm{vol}_{g_{0}} \\
&= \int_{M} \left| \tilde{u}(\cdot, 0) - u \right|^{\frac{2n}{n-2}} \left( \left( 1-\psi \right) + \psi \right)\, d\mathrm{vol}_{g_{0}} \\
&\ge \int_{M} \left| \tilde{u}(\cdot, 0) - u \right|^{\frac{2n}{n-2}} \cdot \psi\, d\mathrm{vol}_{g_{0}} \\
&\ge \left( \frac{\varepsilon}{2} \right)^{\frac{2n}{n-2}} \cdot \mathrm{Vol}(V, g_{0}) > 0
\end{split}
\]
This is a contradiction.

Therefore, from the uniqueness of the normalized Yamabe flow ({\cite[Corollary~3.3]{caldeira2021normalized}} and see Remark \ref{rema-uniqueness}), we obtain that
$\tilde{u}(t) \equiv u(t)$ for all $t \in [0,T].$ 
 
\begin{rema}
One can follow the arguments of Brendle {\cite[Section~2]{brendle2005convergence}}
and we can also obtain $C^{\alpha}$-estimates ({\cite[Proposition~2.6]{brendle2005convergence}}).
However the constant $C$ in {\cite[Lemma~2.5]{brendle2005convergence}} depends not only 
$T, \delta, \kappa, g_{0}$ and $n$ but also 
\[
\int_{M} |R(g_{i}(0)) - r(g_{i}(0))|^{\frac{n^{2}}{2(n-2)}}\, d\mathrm{vol}_{g_{i}(0)},
\]
hence we need an additional information about such an $L^{\frac{n^{2}}{2(n-2)}}$-quantity.
Therefore we have used the Krylov--Safonov estimate \cite{krylov1981certain} instead here.
\end{rema}

\bigskip
Here, we give a proof of Main Theorem \ref{main}. 
(cf. \cite{bamler2016proof} and the author's preprint, arXiv:2208.01865v18)
\begin{proof}[Proof of Main Theorem \ref{main}]
Since $t \mapsto r(g_{i}(t))$ is decreasing and the normalized Yamabe flow (\ref{eq-normalized-yamabe}) preserves the volume, we have
\[
\kappa \ge \int_{M} R(g_{i}(0))\, d\mathrm{vol}_{g_{i}(0)} \ge \int_{M} R(g_{i}(t))\, d\mathrm{vol}_{g_{i}(t)}
\]
for all $t \in [0, T]$ and $i \in \mathbb{N}.$
From the precompactness of the normalized Yamabe flow stated above, (after taking a subsequence) as $i \rightarrow \infty,$ we have
\[
\int_{M} R(g(t))\, d\mathrm{vol}_{g(t)} \le \kappa~~\mathrm{for ~all}~t \in (0, T].
\]
Then, from the $C^{0}$-continuity of $g(t)$ and $R(g(t))$ as $t \rightarrow 0$ (see Remark \ref{rema-prop2.1}),
\[
\int_{M} R(g)\, d\mathrm{vol}_{g} = \int_{M} R(g(0))\, d\mathrm{vol}_{g(0)} 
= \lim_{t \rightarrow 0} \int_{M} R(g(t))\, d\mathrm{vol}_{g(t)} \le \kappa.
\]
\end{proof}

\begin{rema}
\begin{itemize}
\item[(1)] Since 
\[
\begin{split}
\int_{M} &R(g_{i}(t))\, d\mathrm{vol}_{g_{i}(t)} \\
&= \int_{M} - u_{i}(t)^{-\frac{n+2}{n-2}} \left( \frac{4(n-1)}{n-2} \Delta_{g_{0}} u_{i}(t) - R(g_{0}) u_{i}(t) \right) \cdot u_{i}(t)^{\frac{2n}{n-2}}\, d\mathrm{vol}_{g_{0}} \\
&=  \int_{M} \left( \frac{4(n-1)}{n-2} |\nabla u_{i}(t)|^{2} + R(g_{0}) u_{i}(t)^{2} \right)\, d\mathrm{vol}_{g_{0}},
\end{split}
\]
we have actually used only the fact that $g_{i}(t)$ converges to $g(t)$ in the locally uniformly $C^{1}$-sense in $M \times (0,T]$
to verify that
$\int_{M} R(g(t))\, d\mathrm{vol}_{g(t)} \le \kappa$ for all $t \in (0, T]$.

\item[(2)] Assume the same as in Main Theorem \ref{main} with $\delta \le 0$ except $(2),$ and assume that $R(g_{0}) \le 0.$
Then we can prove that
\[
R(g) \ge \delta~~\mathrm{on}~M.
\]
In other words, we can give a partial proof of the Gromov's result \cite{gromov2014dirac} for this case.
Indeed, in our proof of Main Theorem \ref{main} above, $g_{i}(t)$ converges to $g(t)$ on $M \times (0, T]$ in the uniformly $C^{2}$-sense and the lower bound of the scalar curvature
$R(g_{i}(t))$ was preserved under the flow for each $i$, which follows from Proposition \ref{prop-pres}.
Therefore, by the same argument of the proof of Main Theorem \ref{main}, we can prove this claim (see also \cite{bamler2016proof} and the author's preprint, arXiv:2208.01865v18).

\item[(3)] In the argument in the Yamabe non-positive case,
we cannot obtain in general that 
the sequence of the solution $u_{i}$ of the unnormalized Yamabe flow can subconverge to  the limit $u(t)$ up to $C^{1}$-sense on the whole space-time $M \times [0,T].$
Indeed, the sequence of metrics constructed in Example 5.3 of the author's preprint: arXiv:2208.01865v18, has a uniform upper bound of the total scalar curvature, however it does not converge to the limiting metric in the $C^{1}$-sense.
\end{itemize}
\end{rema}

\section{Proof of Main Theorem \ref{main2}}

\,\,\,\,\,\,\, As done in the previous section, we can assume that $R(g_{0}) < 0$ (resp. $\le 0$) when $Y(M, g_{0}) < 0$ (resp. $Y(M, g_{0}) \le 0$).
Then, from the proof of Proposition \ref{prop-short} and the argument in the previous section (Yamabe non-positive case), we can also obtain the long-time solution of the unnormalized Yamabe flow:
\begin{equation}\label{eq-yamabe-flow}
\begin{split}
  \frac{\partial}{\partial t} g_{i}(t) &= -R(g_{i}(t))\, g_{i}(t)~~~~\left( t \in [0, \infty) \right)~~\mathrm{and} \\
  \frac{\partial}{\partial t} g(t) &= -R(g(t))\, g(t)~~~~\left( t \in [0, \infty) \right) 
\end{split}
\end{equation}
respectively with $g_{i}(0) = g_{i}$ and $g(0) = g$ respectively.
(Compared with the situation in Section \ref{section-proof}, the term of the mean scalar curvature $r(g_{i}(t))$ does not appear.)
In order to prove that the unique short-time solution obtained by Proposition \ref{prop-short} is actually long-time solution, it is sufficient to prove that the $C^{0}$-norm of the obtained solution is uniformly bounded from above and below in the space-time (cf. \cite{brendle2005convergence, carron2021convergence}).
More precisely, we can obtain such uniform bounds in the same way as we did in ``Yamabe non-positive case'' of the proof of Main Theorem \ref{main}.
However, the volume is not invariant under the flow in this case, hence we use the following volume estimate (especially from below) instead.
\begin{lemm}[Volume estimates along the unnormalized Yamabe flow]
Under the assumption of Main Theorem \ref{main2}, along the unnormalized Yamabe flow $g(t)$ (\ref{eq-yamabe-flow}),
{\small \[
\mathrm{Vol}(M, g(0)) -\frac{n \kappa}{2} t
\le \mathrm{Vol}(M, g(t)) \le \left( -t Y(M, g_{0}) + \mathrm{Vol}(M, g(0))^{\frac{2}{n}} \right)^{\frac{n}{2}}.
\]}
\end{lemm}
\begin{proof}
Along the flow, the volume evolves as 
\[
\frac{\partial}{\partial t} \mathrm{Vol}(M, g(t)) = -\frac{n}{2} \int_{M}R(g(t))\, d\mathrm{vol}_{g(t)}.
\]
On the other hand, one can easily check that $t \mapsto \int_{M} R(g(t))\, d\mathrm{vol}_{g(t)}$ is decreasing (see the proof of Main Theorem \ref{main2} below), hence
\[
\begin{split}
-\frac{n}{2} \int_{M}R(g(t))\, d\mathrm{vol}_{g(t)} &\ge -\frac{n}{2} \int_{M} R(g(0))\, d\mathrm{vol}_{g(0)} \\
&\ge -\frac{n}{2} \kappa.
\end{split}
\]
Here, we have used the assumption $(c)$ of Main Theorem \ref{main2} in the last inequality.
On the other hand, from the definition of $Y(M, g_{0}),$ we have
\[
-\frac{n}{2} \int_{M}R(g(t))\, d\mathrm{vol}_{g(t)} \le -\frac{n}{2} Y(M, g_{0}) \mathrm{Vol}(M, g(t))^{\frac{n-2}{n}}.
\]
Therefore, from these estimates, the continuity of $u(t)$ as $t \rightarrow 0$ and the mean value theorem, we obtain the desired assertion.
\end{proof}
Thus, we can take a positive time $T > 0$ (independent of $i$) and consider the unnormalized Yamabe flows $(g_{i}(t))_{t \in [0,T]}$ and $(g(t))_{t \in [0,T]}.$
We set $g_{i}(t) =: u_{i}(t)^{\frac{4}{n-2}} g_{0}$ and $g(t) =: u(t)^{\frac{4}{n-2}} g_{0},$
then $u_{i}(0) = u_{i}$ and $u(0) = u.$
In order to prove Main Theorem \ref{main2}, we also require the following $L^{1}$-estimate.
\begin{lemm}[cf. {\cite[Theorem 2.2]{takahashi2022infinite}}]
\label{lemm-esti}
Under the above settings, for any nonnegative smooth function $\psi \in C^{\infty}(M),$ we have
\begin{equation}
\label{eq-L1}
\begin{split}
&\left( \int_{M} \psi |u(t) - u_{i}(t)|\, d\mathrm{vol}_{g_{0}} \right)^{\frac{4}{n+2}} \\
&\le \left( \int_{M} \psi |u(0) - u_{i}(0)|\, d\mathrm{vol}_{g_{0}} \right)^{\frac{4}{n+2}} \\
&~~~+ \left( \frac{(n-1)(n+2)}{n-2} \left( 2 C[\psi] \right)^{\frac{4}{n+2}} + \frac{n+2}{4} \left( \int_{M} \psi\, d\mathrm{vol}_{g_{0}} \right)^{\frac{4}{n+2}} \right)\, t,
\end{split}
\end{equation}
where
\[
C[\psi] := \int_{M} |\Delta_{g_{0}} \psi|^{\frac{n+2}{4}} \psi^{-\frac{n-2}{4}}\, d\mathrm{vol}_{g_{0}}.
\]
\end{lemm}
\begin{proof}
Since $u(t)$ and $u_{i}(t)$ satisfy the unnormalized Yamabe flow equation (\ref{eq-yamabe-flow}) and are positive values, we can obtain the following estimate (see {\cite[Proof of Theorem 2.2]{takahashi2022infinite}}):
For any nonnegative smooth function $\psi \in C^{\infty}(M),$
\[
\begin{split}
\frac{d}{dt} \int_{M} \psi |u(t) - &u_{i}(t)|\, d\mathrm{vol}_{g_{0}} \\
&\le \frac{(n-1)(n+2)}{n-2} \int_{M} \Delta_{g_{0}} \psi \left| u(t)^{\frac{n-2}{n+2}} - u_{i}(t)^{\frac{n-2}{n+2}} \right|\, d\mathrm{vol}_{g_{0}} \\
&~~~+ \frac{n+2}{4} \int_{M} \psi \left| R(g_{0}) \right| \left| u(t)^{\frac{n-2}{n+2}} - u_{i}(t)^{\frac{n-2}{n+2}} \right|\, d\mathrm{vol}_{g_{0}}.
\end{split}
\]
then, from the same argument in {\cite[Proof of Theorem 2.2]{takahashi2022infinite}},
we obtain that
\[
\begin{split}
\frac{d}{dt} \int_{M} &\psi |u(t) - u_{i}(t)|\, d\mathrm{vol}_{g_{0}} \\
&\le \frac{(n-1)(n+2)}{n-2} \left( 2 C[\psi] \right)^{\frac{4}{n+2}} \left( \int_{M} \psi |u(t) - u_{i}(t)|\, d\mathrm{vol}_{g_{0}} \right)^{\frac{n-2}{n+2}} \\
&+ \frac{n+2}{4} \left( \int_{M} \psi\, d\mathrm{vol}_{g_{0}} \right)^{\frac{4}{n+2}} \left( \int_{M} \psi |u(t) - u_{i}(t)|\, d\mathrm{vol}_{g_{0}} \right)^{\frac{n-2}{n+2}},
\end{split}
\]
where
\[
C[\psi] := \int_{M} |\Delta_{g_{0}} \psi|^{\frac{n+2}{4}} \psi^{-\frac{n-2}{4}}\, d\mathrm{vol}_{g_{0}}.
\]
From this, we obtain that 
\[
\begin{split}
\frac{d}{dt} &\left( \int_{M} \psi |u(t) - u_{i}(t)|\, d\mathrm{vol}_{g_{0}} \right)^{\frac{4}{n+2}} \\
&\le \frac{(n-1)(n+2)}{n-2} \left( 2 C[\psi] \right)^{\frac{4}{n+2}} 
+ \frac{n+2}{4} \left( \int_{M} \psi\, d\mathrm{vol}_{g_{0}} \right)^{\frac{4}{n+2}}.
\end{split}
\]
Integrating both sides of the previous inequality in time, we finally obtain that 
\[
\begin{split}
&\left( \int_{M} \psi |u(t) - u_{i}(t)|\, d\mathrm{vol}_{g_{0}} \right)^{\frac{4}{n+2}} \\
&\le \left( \int_{M} \psi |u(0) - u_{i}(0)|\, d\mathrm{vol}_{g_{0}} \right)^{\frac{4}{n+2}} \\
&~~~+ \left( \frac{(n-1)(n+2)}{n-2} \left( 2 C[\psi] \right)^{\frac{4}{n+2}} + \frac{n+2}{4} \left( \int_{M} \psi\, d\mathrm{vol}_{g_{0}} \right)^{\frac{4}{n+2}} \right)\, t.
\end{split}
\]
\end{proof}
By the assumption $(a)$ and the same argument in the previous section (Yamabe non-positive case),
we also obtain the following $C^{\alpha}$-estimates.
\begin{lemm}\label{lem-holder}
Under the condition of Main Theorem \ref{main2}, there is a positive constant $C$ depending only on $M, u,,g_{0}, C_{0}$ and a constant $\kappa$
such that for some $\alpha \in (0,1),$
\[
|u(x_{1}, t_{1}) - u(x_{2}, t_{2})| \le C \left( |t_{1} - t_{2}|^{\frac{\alpha}{2}} + d_{g_{0}}(x_{1}, x_{2})^{\alpha} \right)
\]
and 
\[
|u_{i}(x_{1}, t_{1}) - u_{i}(x_{2}, t_{2})| \le C \left( |t_{1} - t_{2}|^{\frac{\alpha}{2}} + d_{g_{0}}(x_{1}, x_{2})^{\alpha} \right)
\]
for all $i.$
\end{lemm}
Combining these, we can show that $u_{i}$ actually converges to $u$ $C^{0}$-uniformly on $M$ as $i \rightarrow \infty.$
\begin{lemm}
\label{lemm-key}
$u_{i} \rightarrow u$ uniformly on $M$ as $i \rightarrow \infty.$
\end{lemm} 
\begin{proof}
We prove this lemma by contradiction.
Suppose that there is $\varepsilon > 0$ and
a sequence $\{ x_{i} \}$ of points in $M$ such that
\[
|u(x_{i}) - u_{i}(x_{i})| \ge \varepsilon.
\]
Fix a sufficiently small positive constant $0 < r_{0} << \min \{ \mathrm{inj}(M, g_{0}), \varepsilon \},$ where $\mathrm{inj}(M, g_{0})$ is the injectivity radius of $(M, g_{0}).$
In particular, 
\[
|u(x_{i}) - u_{i}(x_{i})| \ge \varepsilon > r_{0} > 0.
\]
We will denote some positive constants depending only on $M, g_{0}, u, \kappa$ and $C_{0}$ as the same symbol $C.$
We take a nonnegative test function $\psi \in C^{\infty}(M)$ such that 
\begin{itemize}
\item $\psi$ has compact support contained in the interior of 
\[
\{ x \in M |~d_{g_{0}}(x_{i}, x) \le r_{0} \},
\]
\item $0 \le \psi \le 1$ on $M.$
\end{itemize}
From the $C^{\alpha}$-estimate from Lemma \ref{lem-holder} above and the volume comparison,
the left-hand side of (\ref{eq-L1}) (in the case of $t = 0$) in Lemma \ref{lemm-esti} is bounded from below by
\[
C r_{0}^{\left( 1 + \frac{n}{\alpha} \right) \frac{4}{n+2}}.
\]
On the other hand, the right-hand side of (\ref{eq-L1}) ($t = 0$) is bounded from above by 
\[
\left( \int_{M} |u(0) - u_{i}(0)|\, d\mathrm{vol}_{g_{0}} \right)^{\frac{4}{n+2}}.
\]
However, from the assumption $(b),$ we can take $i$ sufficiently large such that
\[
\left( \int_{M} |u(0) - u_{i}(0)|\, d\mathrm{vol}_{g_{0}} \right)^{\frac{4}{n+2}} < \frac{C}{2} r_{0}^{\left( 1 + \frac{n}{\alpha} \right) \frac{4}{n+2}}.
\]
This is a contradiction.
\end{proof}

\begin{proof}[Proof of Main Theorem \ref{main2}]
Along the unnormalized Yamabe flow $(g(t))_{t \in [0,T)},$ the scalar curvature evolves as 
\[
\frac{\partial}{\partial t} R(g(t)) = (n-1) \Delta_{g(t)} R(g(t)) + R(g(t))^{2}.
\]
Thus, we obtain that for all $t \in (0,T),$
\[
\begin{split}
\frac{d}{dt} &\left( \int_{M} R(g(t))\, d\mathrm{vol}_{g(t)} \right) \\
&= \int_{M} \left( (n-1)\Delta_{g(t)} R(g(t)) + R(g(t))^{2} - \frac{n}{2} R(g(t))^{2} \right)\, d\mathrm{vol}_{g(t)} \\
&= -\frac{n-2}{2}\int_{M} R(g(t))^{2}\, d\mathrm{vol}_{g(t)} \\
&\le 0.
\end{split}
\]
Hence the function $t \mapsto \int_{M} R(g(t))\, d\mathrm{vol}_{g(t)}$ is decreasing for all $t \in [0,T]$ because $R(g(t)) \overset{C^{0}}{\longrightarrow} R(g(0))$ (see Remark \ref{rema-prop2.1}).
Therefore, from Lemma \ref{lemm-key}, we can follow the same argument in the proof of Main Theorem \ref{main} and prove Main Theorem \ref{main2} in the same manner.
\end{proof}

\begin{ques}
Finally, we conclude this paper by raising the following questions.
\begin{itemize}
\item[$(1)$] Is the H\"{o}lder-bound assumption in our main theorems merely technical? In other words, do our main theorems hold without it?

\item[$(2)$] Is there any example of open manifold on which Main Theorems does not hold?
\item[$(3)$] Is the condition ``$\int_{M} R(g)\, dm \le \kappa$'' for Riemannian manifolds with smooth measure $(M, g, dm := e^{-f} d\mathrm{vol})$ $C^{0}$-closed
in the intersection of the appropriate conformal class ``$[(g, f)]$'' and the space of all Riemannain metrics with measures?  
\end{itemize}
\end{ques}

\section{Appendix}
\label{appendix}
We used the following fact in the proof of Main Theorem \ref{main} (in the Yamabe non-positive case).
\begin{prop}[Gronwall's inequality]
\label{prop-app}
Assume that there is a continuous function $\alpha(t)$ and a nonnegative continuous function $\beta(t)$ on $[0, T]$ such that a continuous function $u(t)$ satisfies the following inequality:
\[
u(t) \le \alpha(t) + \int_{0}^{t} \beta(s) u(s)\, ds~~~~\mathrm{for~all}~t \in [0, T].
\]
Then, for all $t \in [0, T],$
\[
u(t) \le \alpha(t) + \int_{0}^{t} \alpha(s) \beta(s) \exp \left( \int_{s}^{t} \beta(r)\, dr \right)\, ds.
\]
\end{prop}
\begin{proof}
Set 
\[
v(s) := \exp \left( -\int_{0}^{s} \beta(r)\, dr \right) \cdot \int_{0}^{s} \beta(r) u(r)\, dr
\]
for $s \in [0, T].$
Then, we can easily check that
\[
\frac{d}{ds} v(s) = \left( u(s) - \int_{0}^{s} \beta(r) u(r)\,dr \right) \beta(s) \exp \left( -\int_{0}^{s} \beta(r)\, dr \right),~~~s \in [0,T].
\]
Since $v(0) = 0,$ using the assumption and integrating the previous inequality, we obtain that
\[
v(t) \le \int_{0}^{t} \alpha(s) \beta(s) \exp \left( -\int_{0}^{s} \beta(r)\, dr \right)\, ds
\]
for all $t \in [0, T].$
Thus, from the definition of $v(t),$
\[
\begin{split}
\int_{0}^{t} \beta(s) u(s)\, ds &= \exp \left( \int_{0}^{t} \beta(r)\, dr \right) v(t) \\
&\le \int_{0}^{t} \alpha(s) \beta(s) \exp \left( \int_{0}^{t} \beta(r)\, dr - \int_{0}^{s} \beta(r)\, dr \right)\, ds \\
&= \int_{0}^{t} \alpha(s) \beta(s) \exp \left( \int_{s}^{t} \beta(r)\, dr \right)\, ds.
\end{split}
\]
Substituting this into the assumed inequality, we can obtain the desired assertion. 
\end{proof}

\bigskip
\noindent
\textit{E-mail adress}:~hamanaka1311558@gmail.com

\smallskip
\noindent
\textsc{Department of Mathematics, Graduate School of Science, The University of Osaka, Toyonaka, Osaka 560-0043, Japan}

\end{document}